\documentclass[11pt]{amsart}
\usepackage{amssymb,amsmath,enumerate,xypic}
\usepackage[mathscr]{euscript}

\newcommand{\R}{\mathscr{R}} 
\newcommand{\p}{\mathscr{P}} 
\newcommand{\ratk}{\mathrm{k}}  

\newcommand{\G}{G} 
\newcommand{\B}{{B_G}} 
\newcommand{\TT}{{T_G}} 

\newcommand{\Z}{\mathbb{Z}}   
\newcommand{\C}{\mathbb{C}}   

\newcommand{\T}{\mathtt{T}}   
\newcommand{\Ttm}{\T_\trip{m}}  

\newcommand{\HH}{\mathscr{H}} 
\newcommand{\I}{\mathscr{I}} 
\newcommand{\triv}{\mathbf{1}} 
\newcommand{\cond}{\mathrm{cond}}

\newcommand{\trip}[1]{{\mathtt{#1}}}
\newcommand{\diag}{\mathrm{diag}}

\newcommand{\Ind}{\mathrm{Ind}} 
\newcommand{\val}{\mathrm{val}} 

\newcommand{\op}{\mathrm{op}}
\newcommand{\Rcd}{\mathtt{R}_{\trip{c},\trip{d}}}
\newcommand{\Rcc}{\mathtt{R}_{\trip{c},\trip{c}}}
\newcommand{\Rsupp}{\mathtt{S}}
\newcommand{\Rsuppcd}{\mathtt{S}_{\trip{c},\trip{d}}}
\newcommand{\Rsuppcc}{\mathtt{S}_{\trip{c},\trip{c}}}

\newcommand{\Sset}{\mathscr{S}}  

\newcommand{\Wcd}{W_{\trip{c},\trip{d}}}
\newcommand{\Wcc}{W_{\trip{c},\trip{c}}}

\newcommand{\Tcd}{\T_{\trip{c},\trip{d}}}

\newcommand{\acd}{\trip{a}({\trip{c},\trip{d}})}
\newcommand{\acc}{\trip{a}({\trip{c},\trip{c}})}

\newcommand{\Xacd}{\mathtt{X}^\trip{a}_{\trip{c,d}}}
\newcommand{\Xacc}{\mathtt{X}^\trip{a}_{\trip{c,c}}}

\theoremstyle{plain}

\newtheorem{theorem}{Theorem}[section]

\newtheorem{proposition}[theorem]{Proposition}

\newtheorem{corollary}[theorem]{Corollary}

\theoremstyle{definition}
\newtheorem{definition}[theorem]{Definition}

\theoremstyle{remark}
\newtheorem{example}[theorem]{Example}

\numberwithin{equation}{section}
\numberwithin{figure}{section}
\numberwithin{table}{section}

\xymatrixrowsep{20pt}
\xymatrixcolsep{10pt}

\begin{document}
\title[Principal Series of $\mathrm{GL}(3)$]{Branching Rules for Ramified Principal Series Representations of $\mathrm{GL}(3)$ over a $p$-adic field}
\author{Peter S. Campbell}
\address{Department of Mathematics, University of Bristol, UK}
\email{peter.campbell@bristol.ac.uk}

\author{Monica Nevins}
\address{Department of Mathematics and Statistics, University of Ottawa, Canada}\email{mnevins@uottawa.ca}
\thanks{This research is supported by grants from NSERC and from the Faculty of Science of the University of Ottawa.  The first author would also like to acknowledge the support of the Centre de Recherches en Math\'ematiques (CRM) while at the University of Ottawa.}

\date{\today}

\begin{abstract}
We decompose the restriction of ramified principal series representations of
the $p$-adic group $\mathrm{GL}(3,\ratk)$ to its maximal compact subgroup $K=\mathrm{GL}(3,\R)$.
Its decomposition is dependent on the degree of ramification of the inducing
characters and can be characterized in terms of filtrations of the Iwahori
subgroup in $K$.  We establish several irreducibility results and
illustrate the decomposition with some examples.
\end{abstract}
\maketitle

\section{Introduction}

The complex representations of $p$-adic algebraic groups are of great
interest, both in their own right and in what they can reveal through
the Langlands program in number theory.   The representation theory of
$p$-adic groups also often mirrors the theory for real Lie groups,
and it is especially interesting to see how analogous results will
develop.

To this end, one goal is to examine the finer structure of
 representations by considering their restrictions to
compact open subgroups.  The theory of types promises that one
can classify representations in the Bernstein decomposition by
identifying among
certain representations of compact open subgroups which ones
they contain.
In contrast, in the theory of
real Lie groups, the \emph{maximal} compact subgroups have a crucial
role, encoding as they do all the topology of the group, and
one classifies irreducible unitary representations by classifying
the irreducible Harish-Chandra modules.
Our interest is to explore to what extent information about
the representations of the $p$-adic group resides in the maximal
compact subgroup.

Representations of compact subgroups of $p$-adic groups are very
tangible at a number of levels.  Firstly,
the representations
of sufficiently small (exponentiable)
compact open subgroups can all be constructed
using Kirillov theory, as shown by Howe~\cite{Howe}.  Secondly,
each compact open subgroup is pro-finite and consequently
its representation theory is largely determined by the representation theory
of Lie groups over finite local rings.  Finally,
any admissible representation of a $p$-adic group
decomposes with finite
multiplicity upon restriction to a compact open subgroup
and so one can expect to recover
information about the original representation by examining these
constituents.

That said, the maximal compact subgroups are not exponentiable so
Howe's theory does not apply; and furthermore
little information is known in general about the representation theory
of Lie groups over local rings.
Therefore one objective of this study is to provide some
interchange between the representation theories of $p$-adic groups
and of Lie groups over local rings.

Let $\ratk$ be a $p$-adic field and denote by $\R$ its integer ring.
In this paper we consider the group $\G = \mathrm{GL}(3,\ratk)$ and let
$K = \mathrm{GL}(3,\R)$ be a maximal compact subgroup.  In \cite{CNu}, the authors
considered unramified principal series representations and showed
how their restriction to $K$ decomposed as per the double cosets
in $K$ of smaller compact open
subgroups $C_\trip{c}$  (defined in Section~\ref{S:unram}).
In \cite{CNu}, the added assumption that the
inducing character was trivial implied that every double
coset supported an intertwining operator of the representation,
an assumption we relax here.

This paper is organized as follows.  In Section~\ref{S:unram} we
set our notation and
 recall some necessary results from \cite{CNu}.
The key calculation for determining the decomposition is the
determination of the double cosets in $C_\trip{c}\backslash K / C_\trip{d}$
which support intertwining operators for the restricted principal
series representation; this is the main result in Section~\ref{S:support}.
We go on to consider questions of irreducibility in Section~\ref{S:counting}
and conclude with several examples to illustrate these decompositions
in Section~\ref{S:examples}.

The question of parameterizing double cosets of $B$ in $K$, and
more generally of the subgroups $C_{(n,n,n)}$ in $K$, has been
visited and solved  by several authors with various goals in mind.
In \cite{Onn} the goal was to look at which Bruhat decompositions
would be independent of the
characteristic of the residue field; the answer was that only $\mathrm{GL}(2,\ratk)$
has this property.  This implies, in particular, that the decomposition
of principal series is essentially independent of $p$ for $\mathrm{GL}(2,\ratk)$ (see \cite{Monica1, Silberger})
but will depend on the properties of the residue field in all other
cases.

Several authors have considered related questions on the decomposition
of representations of $p$-adic groups upon restriction to a maximal
compact subgroup.  These include the works on Silberger
on $\mathrm{GL}(2,\ratk)$~\cite{Silberger}, the second author on $\mathrm{SL}(2,\ratk)$~\cite{Monica1}, and
Bader and Onn~\cite{Bader} on the Grassmann representation of $\mathrm{GL}(n,\ratk)$.
Gregory Hill has also constructed classes of representations of
$\mathrm{GL}(n,\R)$ in \cite{Hill}; a key part of his results was
the determination of the double cosets of the subgroups $C_{(0,j,j)}$ in $K$.


\section{Notation and Background} \label{S:unram}

Let $\ratk$ be a $p$-adic field of characteristic 0 and residual
characteristic $p$.  Let $q$ denote the number of elements in the
residual field of $\ratk$.  We assume throughout that $p>2$ and $q>3$.
Denote the  integer ring of $\ratk$ by $\R$ and the
maximal ideal of $\R$ by $\p$.
Choose a uniformizer $\pi$ and normalize the discrete valuation
on $\ratk$ so that $\val(\pi)=1$.

Let $\G = \mathrm{GL}(3,\ratk)$ and let $K = \mathrm{GL}(3,\R)$.  Write $\TT$ for the diagonal torus
in $\G$ and $\B$ for the upper triangular Borel subgroup.  Write
$T = \TT \cap K$ and $B = \B \cap K$ for their intersections with
$K$.

\subsection{Principal series and Posets} \label{S:principalseries} Let $\chi_\G$ be a character, not necessarily unitary, of the torus $\TT$ and extend it trivially over
the subgroup $\B$; then the (normalized) induced representation
 $\phi_\G = \Ind_\B^\G \chi_\G$ is a principal series
representation of $\G$.  We consider its restriction to $K$.  Writing
$\chi = \chi_\G\vert_{T}$ and $\phi = \phi_\G\vert_K$, we have that
$\phi = \Ind_B^K \chi$  since $K$ is a good maximal compact.  The principal
series representation is called \emph{ramified} if $\chi \neq \triv$.
The unramified case was considered in \cite{CNu}.

Given a ramified character $\chi_\G$ of $\TT$, we may write it as
$\chi_\G = (\chi_1,\chi_2,\chi_3)$ for characters $\chi_{i} : \ratk^{\times} \rightarrow\C^{\times}$.  Recall
that the \emph{conductor} of a character $\chi_i$ of $\ratk^\times$
is the least $m \geq 0$ such that $1+\p^m \subseteq \ker(\chi_i)$;
thus we make the convention that $\cond(\chi_i)=0$ if and only if
$\chi_i\mid_{\R^\times} = \triv$.

The use of
normalized induction implies that $\Ind_\B^\G \chi \simeq \Ind_\B^\G \chi^w$
for any $w$ in the Weyl group of $\G$, so we may reorder the characters
$\chi_i$ in a convenient way.  Moreover, if $\psi$ is a character of
$\ratk^\times$ and $\psi \cdot \chi = (\psi\chi_1,\psi \chi_2,\psi \chi_3)$,
then $\Ind_\B^\G \psi \cdot \chi = (\psi\circ \det) \Ind_\B^\G \chi$.
It follows that we may assume that $\chi_1 = \triv$ and that
$$
0 \leq M =  \cond(\chi_2) \leq \cond(\chi_3) = N.
$$
Then $\cond(\chi_1\chi_2^{-1}) = M$ and we may furthermore assume that
$\cond(\chi_2\chi_3^{-1}) = \cond(\chi_1\chi_3^{-1}) = N$.
Define $\trip{m} = (M,N,N)$.
We will assume throughout that $\chi \neq \triv$, so
in particular $\chi_3 \neq \triv$  and  $N > 0$.

Let $\T = \{ \trip{c} = (c_1,c_2,c_3) \in \Z^3 \colon 0 \leq c_1, c_2 \leq c_3 \leq c_1+c_2\}$ and note that $\trip{m} = (M,N,N) \in \T$.  Then $\T$ is
a poset with $\trip{c} \preceq \trip{d}$ if $c_i \leq d_i$ for all $i$.
We are particularly interested in the subposet $\Ttm = \{ \trip{c} \in \T \colon \trip{c} \succeq \trip{m}\}$.

Given $\trip{c} \in \T$, we define a subgroup $C_\trip{c}$
by
$$
C_\trip{c} = \left[ \begin{matrix} \R & \R & \R \\ \p^{c_1} & \R & \R \\
\p^{c_3} & \p^{c_2} & \R \end{matrix} \right] \cap K.
$$
Then $C_\trip{c} \subseteq C_\trip{d}$ if and only if $\trip{c} \succeq \trip{d}$.

Let $K_n$ denote
the $n$th principal congruence subgroup of $K$, that is, the normal subgroup of $K$
consisting of all those matrices which are equivalent to the identity matrix
modulo $\p^n$.  Then for all $\trip{c} \in \T$ we have $C_\trip{c} \supset K_{c_3}$.

\subsection*{Subrepresentations}
Let $\chi$ be the restriction to $T$ of a character of $\TT$, with the above conventions; then in particular $\chi_1 = \triv$.  If $\trip{c} \in \Ttm$,
then we can extend $\chi$ to a character of $C_\trip{c}$,  denoted $\chi_\trip{c}$ or simply $\chi$ if there is no possibility of confusion.
 Namely, given
$g = (g_{ij}) \in C_\trip{c}$, we define
 $\chi_\trip{c}(g) = \chi_2(g_{22})\chi_3(g_{33})$.
One verifies directly that this is multiplicative exactly when $c_1 \geq M$
and $c_2, c_3 \geq N$.

\begin{definition} \label{D:Uc}
For each $\trip{c}\in \Ttm$,
set $U_\trip{c} = \Ind_{C_\trip{c}}^K \chi_\trip{c}$.
\end{definition}

We have from \cite{CNu} that $\dim U_\trip{c} = (q+1)(q^2+q+1)q^{c_1+c_2+c_3-3}$ if $c_1c_2 >0$ and $\dim U_\trip{c} = (q^2+q+1)q^{2(c_1+c_2-1)}$ if exactly one
of $c_1$ or $c_2$ is zero.
The representation $U_\trip{c}$ is
naturally a subrepresentation of $\phi$; in fact, it
is contained in the subspace of $K_{c_3}$-fixed vectors of $\phi$.
Consequently,
one may also view $U_\trip{c}$ as a representation of the finite
group $K/K_{c_3}$.

If $\trip{c} \succeq \trip{d}$ then we have $U_\trip{d} \subseteq U_\trip{c}$, so set
$$
V_\trip{c} = U_\trip{c} / \sum_{\trip{d} \in \Ttm, \trip{d} \prec \trip{c}} U_\trip{d}.
$$
This quotient can be identified with a summand of $\phi$.   These
summands are the building blocks of the decomposition of $\phi$ that
we wish to study, so let us refine our description of $V_\trip{c}$.

Let $\trip{c} \in \Ttm$.
If $c_1c_2 \neq 0$, define for
each $i \in \{1,2,3\}$ the triple $\trip{c}_{\{i\}} = (c_1-\delta_{i1},c_2-\delta_{i2},c_3-\delta_{i3}) \in \mathbb{Z}^3$; if $c_1=0$ then only consider
$\trip{c}_{\{3\}} = (0,c_2-1,c_2-1)$; and if $c_2=0$ then only consider
$\trip{c}_{\{3\}} = (c_1-1,0,c_1-1)$.  Set $\Sset_\trip{c} = \{ i \colon \trip{c}_{\{i\}} \in \Ttm \}$; then for all $\trip{d} \prec \trip{c}$ such that $\trip{d} \in \Ttm$, there is
some $i\in \Sset_\trip{c}$ such that $\trip{d} \preceq \trip{c}_{\{i\}}$.

Further, let $\trip{c}_\emptyset = \trip{c}$ and for each non-empty
$I \subseteq \Sset_\trip{c}$ define $\trip{c}_I = \max\{ \trip{d} \in \Ttm \colon \trip{d} \preceq \trip{c}_{\{i\}} \text{ for all } i \in I\}$; this set contains
$\trip{m}$ and so is nonempty.
For example,
if $\trip{m} = (0,0,0)$ and $\trip{c} = (2,3,4)$ then $\trip{c}_{\{1,2\}} = (1,2,3)$ since
$(1,2,4) \notin \Ttm$.
We have the following result from \cite{CNu}.

\begin{theorem} \label{T:calc}
For any $\trip{c}\in \Ttm$ we have
\begin{equation} \label{E:G1}
[V_{\trip{c}}] = \sum_{I \subseteq \Sset_{\trip{c}}}  (-1)^{\vert I \vert }[U_{\trip{c}_{I}}],
\end{equation}
where $[V]$ denotes the equivalence class of $V$ in the Grothendieck group of $K$.
\end{theorem}

Since the $U_\trip{c}$ are essentially induced representations of
finite groups, the dimension $\I(U_\trip{c},U_\trip{d})$
of the space of intertwining operators
between $U_\trip{c}$ and $U_\trip{d}$ is equal
$\dim \HH(\chi_\trip{c}, \chi_\trip{d})$ where
$$
\HH(\chi_\trip{c},\chi_\trip{d}) = \{ f \colon K \to \mathbb{C} \colon f(gkg') = \chi_\trip{c}(g) f(k) \chi_\trip{d}(g') \forall g \in C_\trip{c}, g' \in C_{\trip{d}}\}.
$$
As an immediate
corollary of the above theorem we therefore have an effective means of
determining the number of intertwining operators between the various quotients
$V_\trip{c}$.

\begin{corollary} \label{C:intops}
Let $\trip{c},\trip{d} \in \Ttm$.  Then
the dimension of the space of intertwining operators between $V_\trip{c}$
and $V_\trip{d}$ is
$$
\I(V_\trip{c},V_\trip{d}) = \sum_{I \subseteq \Sset_\trip{c}, J \subseteq \Sset_\trip{d}} (-1)^{\vert I \vert + \vert J \vert}\; \I(U_{\trip{c}_I}, U_{\trip{d}_J}).
$$
\end{corollary}

\subsection{Distinguished double coset representatives of $C_\trip{c}\backslash K / C_\trip{d}$} \label{S:doublecosets}

We recall the parametrization of representatives for the double coset space
$C_\trip{c}\backslash K / C_\trip{d}$, as given in \cite{CNu}.

Let $\T^{1} = \{\trip{a} = (a_1,a_2,a_3) \in \Z^3 \colon 1 \leq a_1, a_2 \leq a_3 \}$.  Given $\trip{c} \in \T$, define $\underline{\trip{c}} \in \T \cap \T^{1}$
by $\underline{c}_i = \max\{ c_i, 1\}$ for each $i$.

\begin{definition} \label{D:Tcd}
For any $\trip{c}, \trip{d} \in \T$, set
\begin{equation} \label{E:Tcddef}
\Tcd = \left\{ \trip{a} \in \T^{1} \colon \trip{a} \preceq \underline{\trip{c}}, \trip{a} \preceq \underline{\trip{d}} \; \text{and} \;
a_3 \leq \min\{a_1+\underline{c}_2,\underline{d}_1+a_2\} \right\}
\end{equation}
with the following exceptions:
\begin{equation} \label{E:Tdefexcept}
\Tcd =
\begin{cases}
\{(1,1,1)\} & \text{if $\trip{c}$ or $\trip{d}$ equals $(0,0,0)$;}\\
\left\{ (1,a,a) \colon a \leq \min\{c_2,d_2\} \right\} & \text{if $c_2d_2>0$ and $c_1=d_1=0$; and}\\
\left\{ (a,1,a) \colon a \leq \min\{c_1,d_1\} \right\} & \text{if $c_1d_1>0$ and $c_2=d_2=0$.}\\
\end{cases}
\end{equation}
\end{definition}

Next, for $\trip{a} \in \Tcd$ set $\min\{\trip{a}\} = \min\{a_1,a_2,a_3\}$.
Then we define
\begin{align} \label{E:defacd}
\acd = \max&\{0, \min\{a_1,a_2, a_3-a_1, a_3-a_2, \trip{c}-\trip{a}, \trip{d}-\trip{a},\\
\notag &  a_1+c_2-a_3, d_1+a_2-a_3\}\}
\end{align}
and
\begin{equation} \label{E:defacdprime}
\acd' = \max\{0,\min\{d_3-a_3,c_3-a_3,c_1-a_1,d_2-a_2\}\} \geq \acd.
\end{equation}

For $x \in \R/\p^k$, define $\val(x) = \min\{ \val(y) \colon y+\p^k = x\}$; then
$(\R/\p^k)^\times = \{ x \in \R/\p^k \colon \val(x)=0 \}$. We set
$$
\Xacd = \begin{cases}
(\R/\p^{\acd})^\times & \text{if $a_1+a_2 \neq a_3$;}\\
\bigcup_{i=0}^{\acd'} (1+\pi^i\R^\times) \cap (\R/\p^{\acd+i})^\times \cap (\R/\p^{\acd'})^\times & \text{if $a_1+a_2=a_3$.}
\end{cases}
$$
In other words, in this latter case,
if $\val(x-1)=i> 0$ then $x$ and $y$ represent the same
element of $\Xacd$ exactly when
$\val(x-y) \geq \min\{i+\acd, \acd'\}$.


\begin{definition}
Let $\trip{c},\trip{d} \in \T$.  Enumerate the elements of $W \simeq S^3$ as
$$
W = \{1,s_1,s_2,s_1s_2,s_2s_1,w_0\}
$$
where $s_i$ is the transposition $(i \ \ i+1)$ and $w_0$ is the longest element.  Define a subset $\Wcd$ of $W$
as
$$
\Wcd = \begin{cases}
W & \text{if $\trip{c}, \trip{d} \succeq (1,1,1)$;}\\
\{1,s_1,w_0\} & \text{if $c_1d_1(c_2+d_2) >0$ and $c_2d_2 =0$;}\\
\{1,s_2,w_0\} & \text{if $c_1d_1 = 0$ and  $(c_1+d_1)c_2d_2 > 0$;}\\
\{1,w_0\} & \text{if $c_1c_2=0$ and $d_1d_2=0$ but $(c_1+c_2)(d_1+d_2) > 0$;}\\
\{1\} & \text{if $\trip{c} = (0,0,0)$ or $\trip{d} = (0,0,0)$}.
\end{cases}
$$
\end{definition}

The following theorem is proven in \cite{CNu}.

\begin{proposition}\label{P:doubleCccosets}
Let $\trip{c},\trip{d} \in \T$.
A complete set of distinct double coset representatives $\Rcd$ of $C_\trip{c}\backslash K / C_\trip{d}$ is
$$
\Rcd = \bigcup_{w \in \Wcd} \Rcd^w
$$
where for $w \in \Wcd$ we define $\Rcd^w$ as follows.
\begin{enumerate}[(i)]
\item $\Rcd^1 = \left\{t_{\trip{a},x} = \left[ \begin{matrix} 1 & 0 & 0 \\ \pi^{a_1} & 1 & 0\\
x\pi^{a_3} & \pi^{a_2} & 1\end{matrix} \right] \colon
\trip{a} \in \Tcd, x \in \Xacd\right\}$;

\item 
$\Rcd^{s_1} = \left\{ s_1^{(\alpha,\beta)} =\left[ \begin{matrix} 0 & 1 & 0\\ 1 & 0 & 0 \\ \pi^{\beta} & \pi^{\alpha} & 1 \end{matrix} \right] \colon \begin{array}{rcl}
1 \leq &\alpha& \leq \min\{\underline{d}_2,c_3\}\\
 1 \leq &\beta& \leq \min\{\underline{c}_2,d_3\}\\
 -c_1 \leq &\beta-\alpha& \leq d_1
\end{array} \right\};$

\item 
$\Rcd^{s_2} = \left\{s_2^{(\alpha,\beta)} = \left[ \begin{matrix} 1 & 0 & 0\\ \pi^{\beta} & 0 & 1 \\ \pi^{\alpha} & 1 & 0 \end{matrix} \right] \colon
\begin{array}{rcl}1 \leq &\alpha& \leq \min\{\underline{d}_1,c_3\}\\
 1 \leq &\beta& \leq \min\{\underline{c}_1,d_3\}\\
 -c_2 \leq &\beta-\alpha& \leq d_2 \end{array} \right\};$

\item 
$\Rcd^{s_1s_2} = \left\{ s_1s_2^{(\alpha)}=\left[ \begin{matrix} 0 & 0 & 1\\ 1 & 0 & 0 \\ \pi^{\alpha} & 1 & 0 \end{matrix} \right] \colon 1 \leq \alpha \leq \min\{d_1,c_2\}\right\};$

\item 
$\Rcd^{s_2s_1} = \left\{ s_2s_1^{(\alpha)} =  \left[ \begin{matrix} 0 & 1 & 0\\ 0 & \pi^{\alpha} & 1 \\  1 & 0&0 \end{matrix} \right] \colon 1 \leq \alpha \leq \min\{c_1,d_2\}\right\}$;

\item $\Rcd^{w_0} = \left\{w_0 = \left[ \begin{matrix} 0 & 0 & 1\\ 0 & 1 & 0 \\ 1 & 0 & 0 \end{matrix} \right]\right\}$.
\end{enumerate}
\end{proposition}

\section{Determination of the set of double cosets supporting intertwining operators} \label{S:support}

To understand the space of
intertwining operators of the
finite-dimensional representations
$U_\trip{c}$ given in Definition~\ref{D:Uc}, we construct bases for
the spaces $\HH(\chi_\trip{c},\chi_\trip{d})$.
That is,  for any $\trip{c}, \trip{d} \in \Ttm$
we must identify among the double cosets enumerated in Proposition~\ref{P:doubleCccosets} those which support   intertwining operators of $U_\trip{c}$ with $U_\trip{d}$.  Denote the subset of these cosets by
$\Rsuppcd \subseteq \Rcd$, and write $\mathcal{I}(U_\trip{c},U_\trip{d}) = \dim \HH(\chi_\trip{c}, \chi_\trip{d}) = \vert \Rsuppcd \vert$.  (Note that in the case that $\chi = \triv$, which we continue to exclude, $\Rcd = \Rsuppcd$ and there is nothing to show.)

If $\trip{a} \in \T^{1}$ then define $\trip{a}^{\op} = (a_3-a_2, a_3-a_1, (a_3-a_1)+(a_3-a_2))$.

\begin{theorem} \label{T:supp}
Let $\trip{c},\trip{d} \in \Ttm$, with $\trip{m} \succ (0,0,0)$ as before.
Then a set of representatives for double cosets in $C_\trip{c}\backslash K /C_\trip{d}$ supporting elements of $\HH(\chi_\trip{c},\chi_\trip{d})$
is
$$
\Rsuppcd = \bigcup_{w \in\Wcd} \Rsuppcd^w
$$
where the subsets $\Rsuppcd^w \subseteq \Rcd^w$ are defined as follows.
\begin{enumerate}[(i)]
\item $\Rsuppcd^1$ is the set of all $t_{\trip{a},x}$ with $\trip{a} \in \Tcd$
and $x\in \Xacd$ such that
 one of the following holds:
\begin{enumerate}[(1)]
\item $a_1 \geq M$ and $a_2 \geq N$; or
\item $a_1 < M$ and $a_2 \geq N$ and:
\begin{enumerate}[(a)]
\item $a_1+a_2 < a_3$ and $M \leq \min\{ \trip{c}-\trip{a}, \trip{d}-\trip{a},  c_2+a_1-a_3, d_1+a_2-a_3\}$; or
\item $a_1+a_2 > a_3$ and $M \leq \min\{ \trip{c}-\trip{a}^{\op}, \trip{d}-\trip{a}^{\op}\}$; or
\item $a_1+a_2 = a_3$ and $M \leq \min\{ \trip{c}-\trip{a}, \trip{d}-\trip{a}\}$
and $\val(x-1) \leq \trip{a}(\trip{c,d})' - M$;
\end{enumerate}
or
\item $a_1 \geq N$ and $a_2 < N$ and the same conditions (a),(b),(c) with $M$ replaced by $N$;
\end{enumerate}
\item $\Rsuppcd^{s_1}$ is the set of all  $s_1^{(\alpha,\beta)}$ such that
\begin{eqnarray*}
N \leq& \alpha &\leq \min\{d_2,c_3\}-M,\\
N \leq& \beta  & \leq \min\{c_2,d_3\}-M, \; \textrm{and} \\
M-c_1 \leq &\beta-\alpha &\leq d_1-M;
\end{eqnarray*}
\item $\Rsuppcd^{s_2}$ is the set of all  $s_2^{(\alpha,\beta)}$
such that
\begin{eqnarray*}
N \leq &\alpha& \leq \min\{d_1,c_3\}-N, \\
N \leq &\beta& \leq \min\{c_1,d_3\}-N, \; \textrm{and}\\
N-c_2 \leq &\beta-\alpha& \leq d_2-N;
\end{eqnarray*}
\item 
and $\Rsuppcd^w = \emptyset$ for all other $w \in \Wcd$ and for any $w \notin \Wcd$.
\end{enumerate}
\end{theorem}

\begin{proof}
Let us first show that none of the cosets represented by
elements of  $\Rcd^{s_1s_2} \cup \Rcd^{s_2s_1} \cup \Rcd^{w_0}$
can support  intertwining operators.
Choose an element $b \in \R^\times$ such that $\chi_3(b) \neq 1$.
Set $g = \diag(b,1,1)$ and $g' = \diag(1,1,b)$; these are elements of $C_\trip{c}$ and $C_\trip{d}$ for
any $\trip{c}, \trip{d} \in \Ttm$.  One verifies that
  $g s_1s_2^{(\alpha)} = s_1s_2^{(\alpha)}g'$,
$g w_0=w_0 g'$ and $s_2s_1^{(\alpha)}g = g's_2s_1^{(\alpha)}$, but that $\chi(g) \neq \chi(g')$.
Consequently none of these representatives are in $\Rsuppcd$.

From now on, let us adopt the notational
convention that if $g = (g_{ij}) \in C_\trip{c}$
then $g_{21} = \gamma_{21}\pi^{c_1}$, $g_{32} = \gamma_{32}\pi^{c_2}$
and $g_{31} = \gamma_{31}\pi^{c_3}$; so $g' \in C_\trip{d}$
would have $g_{21}' = \gamma_{21}'\pi^{d_1}$, and so forth.
Moreover, given a coset representative $h \in \Rcd$ and a pair
of elements $g \in C_\trip{c}$ and $g' \in C_\trip{d}$ such
that $g h = hg'$, we will call $(g,g')$ a \emph{coset pair}.

Suppose now that $(g,g') \in C_\trip{c} \times C_\trip{d}$ are
a coset pair for the representative $s_1^{(\alpha,\beta)}$.
We determine directly that the matrix coefficients
of $g$ and $g'$ satisfy
\begin{eqnarray} \label{E:s1equations}
 \notag g_{22} &=& g_{33} -g_{23}\pi^\beta - \gamma_{21}'\pi^{d_1+\alpha-\beta}+\gamma_{32}\pi^{c_2-\beta} - \gamma_{31}'\pi^{d_3-\beta},\\
g_{22}' &=& g_{33} -g_{23}\pi^\beta  -\gamma_{21}\pi^{c_1+\beta-\alpha} - \gamma_{32}'\pi^{d_2-\alpha} + \gamma_{31}\pi^{c_3-\alpha}, \\
\notag g_{33}' &=& g_{33}  - g_{23}\pi^\beta - g_{23}'\pi^\alpha,
\end{eqnarray}
with the remaining coefficients given by
\begin{equation*}
\begin{array}{rclrcl}
g_{11} &=&  g_{22}' - g_{23}'\pi^{\alpha} \quad & g_{11}' &=& g_{22} + g_{23}\pi^\beta \\
g_{12} &=&  \gamma_{21}'\pi^{d_1} - g_{23}'\pi^\beta \quad & g_{12}' &=& g_{23}\pi^\alpha + \gamma_{21}\pi^{c_1} \\
 g_{13} &=& g_{23}' \quad & g_{13}' &=& g_{23}.
\end{array}
\end{equation*}
This allows us to compare
\begin{eqnarray*}
\chi_\trip{c}(g) &=&\chi_2(g_{22})\chi_3(g_{33})\\
         &=& \chi_2(g_{33} -g_{23}\pi^\beta - \gamma_{21}'\pi^{d_1+\alpha-\beta}+\gamma_{32}\pi^{c_2-\beta} - \gamma_{31}'\pi^{d_3-\beta}) \chi_3(g_{33})
\end{eqnarray*}
with
\begin{eqnarray*}
\chi_{\trip{d}}(g') &=& \chi_2(g_{22}')\chi_3(g_{33}') \\
&=&\chi_2(g_{33} -g_{23}\pi^\beta  -\gamma_{21}\pi^{c_1+\beta-\alpha} - \gamma_{32}'\pi^{d_2-\alpha} + \gamma_{31}\pi^{c_3-\alpha})\cdot \\
&& \chi_3(g_{33}  - g_{23}\pi^\beta - g_{23}'\pi^\alpha).
\end{eqnarray*}
It follows that whenever  $M \leq \min\{c_1 - \alpha+\beta, d_1+\alpha-\beta, c_2-\beta, d_2-\alpha, c_3-\alpha, d_3-\beta\}$
and $N \leq \min\{ \alpha, \beta\}$, then $\chi_\trip{c}(g) = \chi_{\trip{d}}(g')$, and so
$s_1^{(\alpha,\beta)} \in \Rsuppcd$.
Conversely, when these inequalities are not satisfied, and additionally $\alpha, \beta \geq 1$,
then we can use the relations above to construct a coset pair $(g,g')$ on which  the characters do not agree.
This proves part (ii); the
proof of part (iii) is analogous and is omitted.

To prove part (i) of the theorem, suppose $t_{\trip{a},x} \in \Rcd^1$ and let
$(g,g') \in C_\trip{c}\times C_\trip{d}$ be a coset pair such that $g t_{\trip{a},x}  = t_{\trip{a},x} g'$.
To simplify notation, set $r_x = \pi^{a_1+a_2}-x\pi^{a_3}$.
One calculates directly that the matrix coefficients of $g$ and $g'$ satisfy the
relation
\begin{align} \label{E:taxrelation}
(g_{12}\pi^{a_1}+g_{13}\pi^{a_1+a_2} -g_{23}\pi^{a_2} )xr_x  =&
- \gamma_{21}x\pi^{c_1+a_2} - \gamma_{21}'r_x\pi^{d_1+a_2-a_3} \\
\notag & + \gamma_{32}r_x\pi^{a_1+c_2-a_3}
+ \gamma_{32}'x\pi^{a_1+d_2} \\
\notag & +\gamma_{31}\pi^{c_3-a_3+a_1+a_2}
 -\gamma_{31}'\pi^{d_3-a_3+a_1+a_2}
\end{align}
and all other matrix coefficients are determined by the equations
\begin{eqnarray*}
g_{11}' &=& g_{22} + g_{23}x\pi^{a_3-a_1} + \gamma_{21}\pi^{c_1-a_1} - \gamma_{21}'\pi^{d_1-a_1}\\
g_{11} &=& g_{11}' -g_{12}\pi^{a_1} - g_{13}x\pi^{a_3} \\
g_{22}' &=& g_{22} -g_{12}\pi^{a_1} - g_{13}\pi^{a_1+a_2} + g_{23}\pi^{a_2}\\
g_{33}' &=& g_{22} - g_{12}r_x\pi^{-a_2} - \gamma_{32}\pi^{c_2-a_2} + \gamma_{32}'\pi^{d_2-a_2}\\
g_{33} &=& g_{33}' -g_{13}r_x + g_{23}\pi^{a_2},
\end{eqnarray*}
together with $g_{12}' = g_{12} +g_{13}\pi^{a_2}$, $g_{13}' = g_{13}$ and $g_{23}' = g_{23} - g_{13}\pi^{a_1}$.
Note that in this case, as opposed to the one for $s_1^{(\alpha,\beta)}$ above, although any solution (with coefficients in $\R$) of \eqref{E:taxrelation}
gives a pair of matrices $(g,g')$ satisfying the relation $g t_{\trip{a},x}  = t_{\trip{a},x} g'$, it 
must be additionally verified that $g$ and $g'$ are invertible in $K$.

Now, given a coset pair $(g,g')$ we have
\begin{equation} \label{E:chic}
\chi_\trip{c}(g) = \chi_2(g_{22}) \chi_3(g_{33}) = \chi_2(g_{22})\chi_3(g_{33}' -g_{13}r_x + g_{23}\pi^{a_2})
\end{equation}
while
\begin{equation} \label{E:chicp}
\chi_\trip{d}(g') = \chi_2(g_{22} -g_{12}\pi^{a_1} - g_{13}\pi^{a_1+a_2} + g_{23}\pi^{a_2})\chi_3(g_{33}').
\end{equation}
Hence these characters agree
whenever  $a_1 \geq M$ and $a_2 \geq N$,
proving part (i)(1).

Now suppose 
 $a_1$ and $a_2$ are both less than $N$.
Choose a pair $(g_{12},g_{23}) \in \R\times \R$ of minimum
valuation satisfying
$g_{23}\pi^{a_2} = g_{12}\pi^{a_1}$, set $g_{13} = \gamma_{21} = \gamma_{21}' = \gamma_{31} = \gamma_{31}' = \gamma_{32} = \gamma_{32}' = 0$
and set $g_{22} = 1$.  These are easily seen to define a coset pair
$(g,g') \in C_\trip{c} \times C_\trip{d}$.
Since
$\val(g_{12}\pi^{a_1})=\val(g_{23}\pi^{a_2}) = \max\{a_1,a_2\} < N = \cond(\chi_3)$,
we have $\chi_\trip{c}(g) \neq \chi_{\trip{d}}(g')$  and it follows that
 $t_{\trip{a},x} \notin \Rsuppcd^1$.

There are exactly two cases left to consider:
when $a_1 < M$ and $a_2 \geq N$, or when $a_1 \geq N$ and $a_2 < N$.
Comparing \eqref{E:chic} and \eqref{E:chicp},
and noting that
$\max\{a_1, a_2\} \leq \min\{a_3, \val(r_x)\}$, we
deduce that
\begin{enumerate}[(A)]
\item if  $a_1 < M$ and  $a_2 \geq N$, then
$t_{\trip{a},x} \in \Rsuppcd^1$ if and only if $\val(g_{12}\pi^{a_1}) \geq M$
for all coset pairs $(g,g')$; and
\item if $a_1 \geq N$ and $a_2 < N$, then
$t_{\trip{a},x} \in \Rsuppcd^1$ if and only if
$\val(g_{23}\pi^{a_2}) \geq N$ for all coset pairs $(g,g')$.
\end{enumerate}
Consider case (A), that is, assume that
$a_1 < M$ and $a_2 \geq N$.

If $\val(g_{12}\pi^{a_1}) \geq a_2 \geq N$
then we are done; otherwise, the term with least valuation on the left
hand side of \eqref{E:taxrelation} is $g_{12}\pi^{a_1}xr_x$.  Comparing
with the right hand side, we deduce
$\val(g_{12}\pi^{a_1}) + \val(r_x) \geq \alpha$ where
\begin{eqnarray*}
\alpha &=&  \min\{c_1+a_2, d_1+a_2-a_3+\val(r_x), a_1+c_2-a_3+\val(r_x), a_1+d_2,\\
\notag && \quad \quad c_3-a_3+a_1+a_2, d_3-a_3+a_1+a_2\}.
\end{eqnarray*}
It follows that if $\alpha\geq M + \val(r_x)$, then $t_{\trip{a},x} \in \Rsuppcd$
by (A) above.  Restating this condition
in the three cases $a_1+a_2<a_3$, $a_1+a_2>a_3$ and $a_1+a_2=a_3$ yields the
conditions described in part (i)(2)(a,b,c) of the theorem.

Conversely, suppose $\alpha < M+\val(r_x)$ and set $g_{13}=g_{23}=0$.  Choose a term of
least possible valuation on the right hand side of \eqref{E:taxrelation}; set
its coefficient (either $\gamma_{ij}$ or $\gamma_{ij}'$, for some $i>j$)
to be $\pi^{a_1-\alpha}$ if $\alpha<a_1+\val(r_x)$ and $1$ otherwise.
Then set the remaining coefficients of the right hand side of \eqref{E:taxrelation}
equal to zero and solve for $g_{12}$, which is now necessarily in $\R^\times$.
Take $g_{22}=1$ and solve for the remaining coefficients.  This results in
a coset pair $(g,g') \in C_\trip{c} \times C_\trip{d}$ such that $\val(g_{12}\pi^{a_1}) < M$, so
by (A) we conclude $t_{\trip{a},x} \notin \Rsuppcd$, as required.

A similar argument establishes condition (i)(3) of the Theorem, following
case (B),
above.
\end{proof}

Let us conclude this section by deriving some
consequences of Theorem~\ref{T:supp}.  The first, which
is immediate, is a convenient restatement of the theorem
in a special case.

\begin{corollary} \label{C:nnnsupp}
Set $\trip{c}=(n,n,n)$ for $n \geq N$.
The space of intertwining operators
of $U_\trip{c} = V_\chi^{K_n}$ with itself has basis parametrized by
$$
\Rsupp_{n} = \bigcup_{w \in W} \Rsupp_n^w
$$
where
\begin{enumerate}[(i)]
\item $\Rsupp_n^1$ is the set of all $t_{\trip{a},x}$ such that $1 \leq a_1, a_2 \leq  a_3 \leq n$, $x \in \mathtt{X}^\trip{a}_{\trip{c,c}}$ and one of conditions (1), (2) or (3) is met:
\begin{enumerate}[(1)]
\item $\trip{a} \succeq \trip{m}$; or
\item $a_1<M$ and $a_2\geq N$ and:
\begin{enumerate}[(a)]
\item $a_1+a_2 < a_3 \leq n-M$, or
\item $a_1+a_2 > a_3$ and $a_1+a_2 \geq M-n + 2a_3$, or
\item $a_1+a_2 = a_3 \leq n-M$ and $\val(r_x) \leq n-M$; or
\end{enumerate}
\item $a_1 \geq N$, $a_2 < N$ and the same conditions (a),(b),(c), with $M$ replaced by $N$, are satisfied;
\end{enumerate}
\item $\Rsupp_n^{s_1} = \{ s_1^{(\alpha,\beta)} \colon N \leq \alpha, \beta \leq n-M\}$;
\item $\Rsupp_n^{s_2} = \{ s_2^{(\alpha,\beta)} \colon N \leq \alpha, \beta \leq n-N\}$;
\item $\Rsupp_n^{s_1s_2} = \Rsupp_n^{s_2s_1} = \Rsupp_n^{w_0} = \emptyset$.
\end{enumerate}
\end{corollary}

Our second corollary will be relevant for the purposes of
calculating $\I(V_\trip{c}, V_\trip{d})$ in Section~\ref{S:counting}.

\begin{corollary} \label{C:consistency}
Suppose that $\trip{c},\trip{d},\trip{c}',\trip{d}' \in \Ttm$
with $\trip{c} \preceq \trip{c}'$ and $\trip{d} \preceq \trip{d}'$.
Then $\Rsuppcd \subseteq \Rsupp_\trip{c',d'}$.
\end{corollary}

\begin{proof}
By identifying $\Xacd$ with a set of coset representatives from $\R^\times$
in a suitable manner, one easily sees that
$\Rcd \subseteq \mathtt{R}_{\trip{c',d'}}$.  Furthermore, it is
clear
that the list of constraints on elements of $\Rsupp$
in Theorem~\ref{T:supp} can only become less constrictive
as $\trip{c}$ or $\trip{d}$ increase.
\end{proof}

In particular, it makes sense to ask, for a given distinguished double
coset representative $g\in \cup_{\trip{c,d}} \Rcd$,
whether there exist $\trip{c},\trip{d} \in \Ttm$
for which $g \in \Rsuppcd$.

\begin{theorem} \label{T:totalsupp}
The double cosets which
support self-intertwining operators
of $U_\trip{c}$, for some $\trip{c} \in \Ttm$, are
represented by
$$
\Rsupp = \bigcup_{\trip{c},\trip{d} \in \Ttm} \Rsuppcd  = \bigcup_{w \in W} \Rsupp^w
$$
where
\begin{enumerate}[(i)]
\item $\Rsupp^1 = \{ t_{\trip{a},x}  \colon a_3 \geq \max\{a_1,a_2\} \geq N, x \in \R^\times\}$;
\item $\Rsupp^{s_1} = \{ s_1^{(\alpha,\beta)} \colon \alpha, \beta \geq N\}$;
\item $\Rsupp^{s_2} = \{ s_2^{(\alpha,\beta)} \colon \alpha, \beta \geq N\}$; and
\item $\Rsupp^{s_1s_2} = \Rsupp^{s_2s_1} = \Rsupp^{w_0} = \emptyset$.
\end{enumerate}
Moreover, up to identifying $t_{\trip{a},x}$ and $t_{\trip{a},y}$
whenever $x$ and $y$ have the same image in $\Xacd$ for $\trip{c}$,
$\trip{d}$ sufficiently large,
these cosets are all distinct.
\end{theorem}

\begin{proof}
This follows from Corollary~\ref{C:nnnsupp}
by allowing $n$ to grow without bound.  Note that when $\trip{c} = \trip{d} = (n,n,n)$, we have simply
$\acd = \min\{ a_1, a_2, a_3-a_1, a_3-a_2, n-a_3\}$ and $\acd' = n-a_3$.
\end{proof}

\section{Irreducibility} \label{S:counting}

The results of the preceding section allow us to restate
Corollary~\ref{C:intops} in terms of the sets $\Rsuppcd$.
That is, for any $\trip{c},\trip{d}\in \Ttm$, we have
$$
\mathcal{I}(V_{\trip{c}}, V_{\trip{d}})
=
\sum_{I\subseteq \Sset_{\trip{c}},\; J\subseteq \Sset_{\trip{d}}}
(-1)^{|I|+|J|} \;
\vert \Rsupp_{\trip{c}_I,\trip{d}_J} \vert.
$$

The irreducibility of $U_\trip{m}=V_\trip{m}$ is known
from Howe's work~\cite[Theorem 1]{Howe1}.  In this section, we demonstrate
that this extends to many, but not all, of the quotients which are ``extremal'' in the sense that they have few immediate descendants in the poset $\Ttm$.

We retain the notation of the previous sections.

\begin{theorem} \label{T:irred}
For each $n \in \mathbb{Z}$ with $N \leq n \leq N+M$, the
$K$-module $V_{(M,N,n)}$
is irreducible.
\end{theorem}

\begin{proof}
Set $\trip{c} = (M,N,n)$.  If $n=N$ then $\Sset_\trip{c}=\emptyset$; otherwise,
$\Sset_\trip{c}$ is a singleton corresponding to the triple $(M,N,n-1)$.
By Corollary~\ref{C:intops}, and induction,
 it thus suffices to show
that $\vert \Rsupp_{\trip{c,c}} \vert = n-N+1$.

If $M=0$ then $\Wcc = \{1, w_0\}$ so $\Rsuppcc^{s_1} = \Rsuppcc^{s_2} = \emptyset$.  
If $M>0$ then
$\min\{c_2,c_3\}-M = N-M<N$ and
$\min\{ c_1,c_3\}-N=M-N<N$ so again
$\Rsuppcc^{s_1} = \Rsuppcc^{s_2} = \emptyset$,
regardless of the value of $n$.  Thus $\Rsuppcc = \Rsuppcc^1$.

Now let $t_{\trip{a},x} \in \Rsuppcc^1$; so one of Theorem~\ref{T:supp}(i)(1), (2) or (3) applies.  If it were (2), then $a_1 < M$ and $a_2 \geq N$
imply that $M>0$ and  $a_2 = N$ so neither case (a) nor case (c) could
apply since $M > c_2-a_2 = 0$.  Were case (b) to apply, then
$M \leq \min\{ \trip{c-a}^\op\}$ would imply that $a_2=a_3 = N$
and so $N-(a_3-a_1) = a_1$ which is not greater than or equal to $M$,
a contradiction.  We similarly deduce that case (3) cannot apply.
This leaves case (1), which consists of the elements
$t_{\trip{a},x}$
with $\trip{a} = (M,N,m)$, $N \leq m \leq n$ and $x \in \Xacc$, each of
which support an intertwining
operator.   For each such $\trip{a}$, we have $\acc = \acc' = 0$, so
in fact  $\vert \Xacc \vert = 1$.
The desired conclusion follows.
\end{proof}

\begin{theorem} \label{T:otherirred1}
$V_{(m,n,n+m)}$ is irreducible for each $m \geq M$ and  $n \geq N$.
\end{theorem}

\begin{proof}
We first consider the case that $\trip{c} = (m,n,m+n)$ with $m \geq 1$.  Then $\Sset_{\trip{c}}$ is a singleton corresponding to
$\trip{d}=\trip{c}_{\{3\}}=(m,n,m+n-1)$.
Hence $\I(U_\trip{c},V_\trip{c}) = \I(U_\trip{c},U_\trip{c}) -\I(U_\trip{c},U_\trip{d})$ and it suffices to show that
 $\vert \Rsupp_{\trip{c,c}} \setminus \Rsuppcd\vert =1$.

First note that since $c_1=d_1$ and $c_2=d_2$, and that both
these are at most $d_3<c_3$, we have
$\Rsupp_{\trip{c,c}}^w = \Rsuppcd^w$ for each $w \in W \setminus \{1\}$.

Next, note that $\T_\trip{c,c} \setminus \Tcd$ consists of the
single element $(m,n,n+m) = \trip{c}$.  Since
$\vert \mathtt{X}^\trip{c}_{\trip{c,c}} \vert = 1$,
there is a unique distinguished double coset of the form $t_{\trip{c},x} \in \mathtt{R}^1_{\trip{c,c}}$; it is clearly in $\Rsupp_{\trip{c,c}}$.
We claim that this is the only element of
$\Rcc^1 \setminus \Rcd^1$.  Namely, let
$\trip{a} \in \Tcd$.
Since $0 \leq d_3-a_3 = (d_1+d_2-1)-a_3 = (d_1+a_2-a_3)+(d_2-a_2)-1$,
it must be true that  $d_3-a_3 \geq \min\{d_1+a_2-a_3, d_2-a_2\}$
so necessarily
$\trip{a(c,c)}=\trip{a(c,d)}$.
If furthermore $a_1+a_2=a_3$,
then $\trip{a(c,c)}'=\trip{a(c,d)}'$ by the same reasoning.
Hence $\mathtt{X}^\trip{a}_\trip{c,c} = \Xacd$ for all such $\trip{a}$,
as claimed.

We next claim that $\Rsupp_{\trip{c,c}}^1 \cap \Rcd^1 = \Rsuppcd^1$.
Namely, given $t_{\trip{a},x} \in \Rsupp_{\trip{c,c}}^1 \cap \Rcd^1$,
since $d_3-a_3 \geq \min\{d_1+a_2-a_3, d_2-a_2\} = \min\{c_1+a_2-a_3, c_2-a_2\}$,
we see that all of the conditions set out in Theorem~\ref{T:supp}(i)
are unchanged in passing from the pair $(\trip{c,c})$ to the
pair $(\trip{c,d})$.  Hence $t_{\trip{a},x} \in \Rsuppcd^1$.

This shows, for the case $m\geq 1$, that  $\Rsupp_{\trip{c,c}}\setminus \Rsuppcd$
is a singleton from which we deduce the irreducibility of $V_\trip{c}$.

When $\trip{c} = (0,n,n)$, we have instead
$\Sset_\trip{c} = \{3\}$ corresponding to $\trip{d} = \trip{c}_{\{3\}} = (0,n-1,n-1)$.  We have $\Rsuppcd = \Rsuppcd^1$, $\Rsupp_{\trip{c,c}} = \Rsupp_{\trip{c,c}}^1$ and neither of the cases (i)(2) or (i)(3) of Theorem~\ref{T:supp} can apply.  It thus follows easily that  $\vert \Rsupp_{\trip{c,c}}\setminus \Rsuppcd \vert =1$ in this case as well.
\end{proof}

\begin{theorem} \label{T:otherirred2}
$V_{(m,n,\max\{n,m\})}$ is irreducible for each $m > M$ and $n > N$.
\end{theorem}

\begin{proof}
Suppose first that $m > \max\{M,1\}$ and $n > N$, and
that $\max\{m,n\} = n$.
Then $\trip{c} = (m,n,n)$
and  $\Sset_\trip{c}= \{1,2\}$ with the corresponding
triples
$\trip{c}_{\{1\}}= (m-1,n,n)$,
$\trip{c}_{\{2\}} = (m,n-1,n)$ and
$\trip{c}_{\{1,2\}} = (m-1,n-1,n)$.
We compute the alternating sum
\begin{equation} \label{E:icalc}
\mathcal{I}(U_\trip{c},V_\trip{c})=
\mathcal{I}(U_\trip{c},U_\trip{c})
- \mathcal{I}(U_\trip{c},U_{\trip{c}_{\{1\}}})
- \mathcal{I}(U_\trip{c},U_{\trip{c}_{\{2\}}})
+ \mathcal{I}(U_\trip{c},U_{\trip{c}_{\{1,2\}}})
\end{equation}
as a sum of differences by defining
\begin{eqnarray*}
\mathcal{A}_0&=&\Rsupp_{\trip{c,c}} \setminus \Rsupp_{\trip{c},\trip{c}_{\{2\}}}\\
\mathcal{A}_1&=&\Rsupp_{\trip{c},\trip{c}_{\{1\}}} \setminus \Rsupp_{\trip{c},\trip{c}_{\{1,2\}}}.
\end{eqnarray*}
Thus we have
$\mathcal{I}(U_\trip{c},V_\trip{c})=\vert \mathcal{A}_0 \vert - \vert
\mathcal{A}_1 \vert$.   We use $(\trip{d,d}')$ to denote either of the
pairs $(\trip{c},\trip{c}_{\{2\}})$ or $(\trip{c}_{\{1\}},\trip{c}_{\{1,2\}})$,
for ease of notation.

Suppose
first that $s_1^{(\alpha,\beta)} \in \Rsuppcd\setminus\Rsupp_\trip{c,d'}$.
Then, comparing the constraints on $\alpha$ and $\beta$
in Theorem~\ref{T:supp}(ii) for  $\trip{d}$ and $\trip{d}'$,
we see that necessarily
$\alpha =d_2-M=n-M$ and
$$
\max\{N,M-c_1+\alpha\}\leq \beta
\leq \min\{d_1-M+\alpha,c_2-M\}.
$$
Since $d_1 \geq M$ by hypothesis, $c_2-M= n-M\leq \alpha+ d_1-M$ so
these inequalities simplify to $\max\{N,n-m\} \leq \beta \leq n-M$.
This constraint on the pair $(\alpha,\beta)$ is independent
of the value of $d_1 \in \{m-1,m\}$ so
$s_1^{(\alpha,\beta)} \in \mathcal{A}_0$ if and only if
$s_1^{(\alpha,\beta)} \in \mathcal{A}_1$.  Hence these cosets
contribute nothing to the overall sum (\ref{E:icalc}).

Now suppose that $s_2^{(\alpha,\beta)} \in \Rsuppcd$.
Then Theorem~\ref{T:supp}(iii)
implies that $\beta - \alpha \leq (c_1-N)-N$;
but this bound is at most $d_2'-N$ since $c_1-N = n-N
\leq n-1 = d_2'$.  Similarly, $d_1 \leq c_1$ and $\beta > 0$ together
imply that $\beta-\alpha \geq N-c_1$, regardless of the value of $d_1\in \{m-1,m\}$.
All other conditions on $(\alpha, \beta)$ being unchanged in passing
from $(\trip{c},\trip{d})$ to $(\trip{c},\trip{d'})$,
we deduce that  $s_2^{(\alpha,\beta)} \in \Rsupp_{\trip{c,d'}}$.
Hence none of these cosets appear in either $\mathcal{A}_0$ or $\mathcal{A}_1$.

Finally, consider distinguished coset representatives of the form
$t_{\trip{a},x} \in \Rcd^1$.  First note that
$$
\Tcd\setminus \T_{\trip{c,d'}} = \{(a_1,n,n) \colon 1 \leq a_1 \leq \underline{d}_1\},
$$
and $\vert \Xacd \vert  = \vert \mathtt{X}^\trip{a}_{\trip{c,d'}}\vert  = 1$, since $a_3-a_2=0$.
Considering which of these are in $\Rsuppcd$, we deduce that these triples give rise to $m-M+1$ coset representatives in $\mathcal{A}_0$ and $m-M$ of
them in $\mathcal{A}_1$.

Suppose now that $\trip{a} \in \T_\trip{c,d'}$.
Since $0 \leq d_2'-a_2 = d_3-1-a_2 = (d_3-a_3)+(a_3-a_2)-1$, we
have  $d_2'-a_2 \geq \min\{d_3-a_3,a_3-a_2\}$ and so
it follows that $\trip{a(c,d)} = \trip{a(c,d')}$.
Similarly, if $a_1+a_2=a_3$ then
$a_2<a_3$ implies that $d_2'-a_2 \geq d_3-a_3$ so
$\trip{a(c,d)}'=\trip{a(c,d')}'$.  Hence for all $\trip{a} \in \T_\trip{c,d'}$ we have
$\Xacd = \mathtt{X}^\trip{a}_{\trip{c,d}'}$.

So suppose $t_{\trip{a},x} \in \Rsuppcd^1 \cap \mathtt{R}_\trip{c,d'}^1$.
We first note that if $t_{\trip{a},x}$ falls under any of the conditions
(2a), (2c), (3a) or (3c) of Theorem~\ref{T:supp}, then the
inequality $a_2<a_3$ implies $d_2' - a_2 \geq d_3-a_3$.  Consequently, this
condition is unchanged  in passing from $\trip{d}$ to $\trip{d}'$ and
so $t_{\trip{a},x} \in \Rsupp_\trip{c,d'}$.  Similarly, if $t_{\trip{a},x}$
falls under condition (3b), then $a_2<a_3$ so $d_2'-(a_3-a_1) \geq d_3-(a_3-a_1)-(a_3-a_2)$;
again we deduce $t_{\trip{a},x} \in \Rsupp_\trip{c,d'}$.  So none of
these occur in either $\mathcal{A}_0$ or $\mathcal{A}_1$.

On the other hand, if $t_{\trip{a},x}$
falls under condition (2b) for the pair $(\trip{c,d})$,
then it fails (2b) for the pair $(\trip{c,d'})$
exactly when
$a_3=a_2 \geq N$, $d_1 \geq M$, $d_2-(a_3-a_1)=M$ and
$1 \leq a_1 < M$.  Hence, noting also that this condition is
independent of the choice of $d_1 \in \{m-1,m\}$,
 all such $t_{\trip{a},x}$ lie in both $\mathcal{A}_0$
and $\mathcal{A}_1$.

We deduce that $\vert \mathcal{A}_0 \vert - \vert
\mathcal{A}_1 \vert$ = 1 so the quotient $V_\trip{c}$ is
indeed irreducible.

The case for $m \geq n$ follows by an analogous argument,
where we interchange the roles of $\trip{c}_{\{1\}}$ and $\trip{c}_{\{2\}}$
throughout.

It only remains to show the case where $m = 1$ and  $M=0$.  In this case,
$\trip{c} = (1,n,n)$ with $n > N \geq 1$ so we have
$\trip{c}_{\{1\}} = (0,n,n)$,
$\trip{c}_{\{2\}} = (1,n-1,n)$ and
$\trip{c}_{\{1,2\}} = (0,n-1,n-1)$.
Define $\mathcal{A}_0$ and $\mathcal{A}_1$ as above.  Since $c_1-a_1 = 0$
for all $\trip{a} \in \Tcd$, for any $\trip{d}$, and since cases (i)(2) and (i)(3) cannot
occur, the analysis is much simplified from the above.  We readily
see that $\mathcal{A}_0 = \{ s_1^{(n,n-1)}, s_1^{(n,n)}, t_{(1,n,n),1} \}$
whereas $\mathcal{A}_1 = \{ t_{(1,n,n),1}, t_{(1,n-1,n),1} \}$.
Thus we  conclude again in this case that $V_\trip{c}$ is irreducible.
\end{proof}

Recalling that $V^{K_n} \simeq U_{(n,n,n)}$ we immediately have the
following Corollary.

\begin{corollary} \label{C:maximal}
If $n>N$ then the quotient $V_{(n,n,n)}$ is the unique irreducible
of maximal dimension in $V^{K_n}$.
\end{corollary}

The strict inequalities in Theorem~\ref{T:otherirred2}
are necessary, as the following proposition shows.

\begin{proposition} \label{P:red} \
\begin{enumerate}[(1)]
\item Let $\trip{c} = (M,n,n)$ with $n > N$.  Then
$$
\I(V_\trip{c},V_\trip{c}) = \begin{cases}
n-N+1 & \text{if $n < M+N$;}\\
M+1 & \text{if $n \geq M+N$}.
\end{cases}
$$
\item Let $\trip{c} = (n,N,n)$ with $n \geq N$.  Then
$$
\I(V_\trip{c},V_\trip{c}) = \begin{cases}
n-N+1 & \text{if $n < 2N$};\\
N+1 & \text{if $n \geq 2N$}.
\end{cases}
$$
\end{enumerate}
\end{proposition}

\begin{proof}
To prove part (1), let $\trip{c} = (M,n,n)$ with $M > 0$ and $n > N$.  Then $\Sset_\trip{c} = \{2\}$
corresponding to the triple $\trip{c}_{\{2\}} = (M,n-1,n)$.  We first
compute $\I(U_\trip{c}, V_\trip{c}) = \vert \Rsupp_{\trip{c,c}} \setminus
\Rsupp_{\trip{c,c_{\{2\}}}} \vert$.  For ease of notation,
set $(\trip{d},\trip{d}') = (\trip{c},\trip{c}_{\{2\}})$.

It is easy to see that
$s_1^{(n-M,n-M)} \in \Rsuppcd \setminus \Rsupp_{\trip{c,d}'}$
if $n-M \geq N$, whereas $\Rsuppcd^{s_2} = \Rsupp_{\trip{c,d}'}^{s_2}=\emptyset$.

Of the elements in $\Tcd \setminus \T_{\trip{c,d'}} = \{ (a_1,n,n) \colon 1 \leq a_1 \leq M\}$, only $(M,n,n)$ gives rise to
a representative in $\Rsuppcd$, and
then exactly one, which we'll denote $t_{(M,n,n),1}$.

For each $\trip{a} \in \T_{\trip{c,d'}}$, we have $a_2 < n$ and $a_3 \leq n$.
Since $\acd \leq \min\{a_3-a_2, n-a_3, d_2-a_2\}$ and $a_3 \leq d_2$
we deduce that if $\acd \neq \trip{a}(\trip{c,d}')$ then necessarily $a_3=n$
and $a_3=a_2$, a contradiction.  Similarly, $\acd'$ does not depend on the
value of $d_2$.  Hence $\Xacd = \mathtt{X}^\trip{a}_{\trip{c,d'}}$.


So suppose $t_{\trip{a},x} \in (\Rsuppcd^1 \cap \mathtt{R}_{\trip{c,d'}}) \setminus \Rsupp_{\trip{c,d'}}$.  It does not fall under case (i)(1)
of Theorem~\ref{T:supp} since this case
is independent of $\trip{d}'$; nor can case (i)(3) occur since $a_1 \leq M$.
In cases (i)(2)(a) and (c), the right hand side can depend on the value of
$d_2 \in \{n-1,n\}$ if and only if $a_2 = a_3$, contradicting the hypotheses.
In case (i)(2)(b), which holds only if $a_2 \geq N$,
we must have that $a_3-a_2 = 0$ or else the right hand
side is less than $M$.  It follows that the right hand side depends on
the value of $d_2$ exactly when $a_2=a_3\geq N$ and $n - (a_3-a_1) = M$;
in each of these cases $\vert \Xacd\vert = \vert \mathtt{X}^\trip{a}_{\trip{c,d'}} \vert =1$ and $t_{\trip{a},1} \in \Rsuppcd \setminus \Rsupp_{\trip{c,d'}}$.

We conclude that when $M>0$
\begin{align*}
\Rsupp_\trip{c,c} \setminus \Rsupp_{\trip{c,c}_{\{2\}}} &=
\{ t_{(M-k,n-k,n-k),1} \colon 0 \leq k \leq \min\{M-1,n-N\} \}\\
& \quad  \cup \{ s_1^{(n-M,n-M)} \colon \text{if $n-M \geq N$}\}.
\end{align*}
A simpler analysis, which we consequently omit, allows us to further deduce that $\Rsupp_{\trip{c}_{\{2\}},\trip{c}} = \Rsupp_{{\trip{c}_{\{2\}}},{\trip{c}_{\{2\}}}}$
and so $\I(U_\trip{c},V_\trip{c}) = \I(V_\trip{c},V_\trip{c})$, and this has the value stated in the theorem.

When $M=0$, we have instead $\trip{c} = (0,n,n)$ and $\trip{c}_{\{3\}} = (0,n-1,n-1)$, and $\Rsuppcd = \Rsuppcd^1$.  Since neither (i)(2) nor (i)(3) of Theorem~\ref{T:supp} can apply, and
$\trip{a}(\trip{d,d'}) =  \trip{a}(\trip{d,d'})'=0$ for all choices of $\trip{d},\trip{d}' \in \{ \trip{c}, \trip{c}_{\{3\}}\}$ and for all $\trip{a} \in \T_{\trip{d,d}'}$,
we readily conclude that $\Rsupp_{\trip{c,c}} \setminus \Rsupp_{\trip{c,c}_{\{3\}}} = \{ t_{(1,n,n),1} \}$.  Hence the quotient
$V_\trip{c}$ is irreducible in this case.

To prove part (2), let $\trip{c} = (n,N,n)$ with $n \geq N$.  Then $\Sset_\trip{c} = \{1\}$ with corresponding triple $(n-1,N,n)$.  Reasoning as above, we deduce
readily that $s_2^{(n-N,n-N)} \in \Rsupp_{\trip{c,c}} \setminus \Rsupp_{\trip{c,c}_{\{1\}}}$ whenever $n\geq 2N$ and
that $\Rsupp^{s_1}_{\trip{c,c}} = \Rsupp^{s_1}_{\trip{c,c}_{\{1\}}}=\emptyset$.

Set $(\trip{d},\trip{d}') = (\trip{c},\trip{c}_{\{1\}})$.
Note that $\Tcd \setminus \T_{\trip{c,d}'} = \{(n,a_2,n) \colon 1 \leq a_2 \leq N\}$ and each of
these has $\vert \Xacd \vert = 1$.  These triples thus give rise to only one
coset in $\Rsuppcd^1$, namely that represented by $t_{(n,N,n),1}$.

Of those $\trip{a} \in  \T_{\trip{c,d}'}$, one sees as above that $\Xacd = \mathtt{X}^\trip{a}_{\trip{c,d'}}$.   For such a triple $\trip{a}$, if
$t_{\trip{a},x} \in \Rsuppcd \setminus \Rsupp_{\trip{c,d'}}$ then
it falls under case (i)(3)(b) of Theorem~\ref{T:supp} and
we deduce as above that $N \leq a_1=a_3 \leq n$ and
$n-a_3 = N-a_2$.  These conditions further imply that $\vert \Xacd \vert = 1$
meaning each such triple gives rise to a unique double coset.

We conclude that
\begin{align*}
 \Rsupp_\trip{c,c} \setminus \Rsupp_{\trip{c,c}_{\{1\}}} &=
\{ t_{(n-k,N-k,n-k),1} \colon 0 \leq k \leq \min\{n-N,N-1\}\} \\
& \quad
\cup \{ s_2^{(n-N,n-N)} \colon \text{if $n-N \geq N$}\}.
\end{align*}
It is readily verified that
$\Rsupp_{\trip{c}_{\{1\}},\trip{c}} = \Rsupp_{{\trip{c}_{\{1\}}},{\trip{c}_{\{1\}}}}$,
and so $\I(U_\trip{c},V_\trip{c}) = \I(V_\trip{c},V_\trip{c})$.  Counting
the double cosets in the expression above yields part (2) of the theorem.
\end{proof}

\section{Examples} \label{S:examples}

We conclude the paper with two examples to illustrate
the results in Section~\ref{S:counting}.

\begin{example} \label{Ex:1}

\begin{figure}[th]
$$\xymatrix{
         &         &      \underset{1}{(4,4,4)}\ar[dr]\ar[dl]  \\
       & \underset{1}{(3,4,4)}\ar[dr]\ar[dl]&               & \underset{1}{(4,3,4)}\ar[dr]\ar[dl]  \\
\underset{3}{(2,4,4)}\ar[d]&&\underset{q-1}{(3,3,4)}\ar[drr]\ar[dll]\ar[d]  &    & \underset{3}{(4,2,4)}\ar[d] \\
 \underset{1}{(2,3,4)}\ar[d]\ar[drr]&& \underset{1}{(3,3,3)}\ar[drr]\ar[dll] && \underset{1}{(3,2,4)}\ar[d]\ar[dll] \\
 \underset{2}{(2,3,3)}\ar[drr]&&  \underset{1}{(2,2,4)}\ar[d]&    & \underset{2}{(3,2,3)}\ar[dll] \\
         && \underset{1}{(2,2,3)}\ar[d]  \\
     && \underset{1}{(2,2,2)}\\
}$$
\begin{caption}{\protect{\label{F:example1}} Reducibility of $V_\trip{c}$ for $V_\chi$ with  $M=2$ and $N=2$. }
\end{caption}
\end{figure}

\begin{table}[th]
\begin{tabular}{|ll|}
\hline
Quotient & Dimension \\ \hline
$V_{(4,4,4)}$ & $q^7(q-1)^2\alpha$\\
$V_{(3,4,4)}, V_{(4,3,4)}$ & $q^6(q-1)^2\alpha$\\
$V_{(3,3,4)}$ & $q^4(q-1)^3\alpha$\\
$V_{(2,4,4)}, V_{(4,2,4)}$ & $q^6(q-1)\alpha$\\
$V_{(2,3,4)}, V_{(3,3,3)}, V_{(3,2,4)}$ & $q^4(q-1)^2\alpha$\\
$V_{(2,3,3)}, V_{(2,2,4)}, V_{(3,2,3)}$ & $q^4(q-1)\alpha$\\
$V_{(2,2,3)}$ & $q^3(q-1)\alpha$\\
$V_{(2,2,2)}$ & $q^3\alpha$\\ \hline
\multicolumn{2}{c}{\ } \\
\end{tabular}
\begin{caption}{\protect{\label{T:example1}}
Dimensions of $V_\trip{c}$ for $V_\chi$ with $M=2$ and $N=2$.}
\end{caption}
\end{table}

Suppose that $M=N=2$ and let us consider the decomposition of
$V^{K_4}$ under $K$.  The values of $\mathcal{I}(V_\trip{c},V_\trip{d})$
are calculated using Corollary~\ref{C:intops} and Theorem~\ref{T:supp},
with several values identified by Theorems~\ref{T:irred}, \ref{T:otherirred1}
and \ref{T:otherirred2} and Proposition~\ref{P:red}.
The remaining computations were implemented in GAP~\cite{GAP}
and the results are
represented schematically in Figure~\ref{F:example1}, as follows.

 Each triple $\trip{c}$ in Figure~\ref{F:example1}
corresponds to the induced representation $U_\trip{c}$, and the
number beneath it is the value of $\mathcal{I}(V_\trip{c},V_\trip{c})$.  The arrows imply
the partial order $\succeq$ on $\T$; hence the set of all
components of the diagram below and including $\trip{c}$ may be identified with
the whole of $U_\trip{c}$.
For reference, we list in Table~\ref{T:example1} the dimensions of the quotients
$V_\trip{c}$ occuring in Figure~\ref{F:example1}.  These are calculated using Theorem~\ref{T:calc}.  We abbreviate $\alpha = (q+1)(q^2+q+1)$.

Figure~\ref{F:example1} reveals several typical features of the $K$-rep\-re\-sen\-ta\-tions
$V_\trip{c}$.
For example, we note that while many $V_\trip{c}$ are irreducible, several are
not.  Besides those identified by Proposition~\ref{P:red}, for whom the number of intertwining operators grows at most linearly with $c_3$,
there exist components such
as $V_{(3,3,4)}$, for which the number of intertwining operators
is a polynomial function of $q$.  Such components
occur more frequently in $V^{K_n}$ as $n$ increases, since they come
into  existence only when $\vert \mathtt{X}^\trip{a}_\trip{c,c} \vert$ is a
polynomial in $q$, that is, when $\trip{a}(\trip{c,c}) >0$.
\end{example}

\begin{example} \label{Ex:2}

\begin{figure}[t]
$$\xymatrix{
 \underset{2\; \dagger}{(1,4,4)}\ar[d]  & \underset{q-2 \; \; \dagger}{(2,3,4)}\ar[dr]\ar[dl]\ar[d] & \underset{1}{(3,3,3)}\ar[dr]\ar[dl] & \underset{1}{(3,2,4)}\ar[d]\ar[dl]        \\
 \underset{1\; *}{(1,3,4)}\ar[d] & \underset{1}{(2,3,3)}\ar[dr]\ar[dl]  & \underset{1 \; *}{(2,2,4)}\ar[d] & \underset{2}{(3,2,3)}\ar[dl]  \\
 \underset{2}{(1,3,3)}\ar[d]  & & \underset{1}{(2,2,3)}\ar[dll]\ar[d] \\
 \underset{1}{(1,2,3)}\ar[dr]  & & \underset{1}{(2,2,2)}\ar[dl] \\
& \underset{1}{(1,2,2)}\\
}$$
\begin{caption}{\protect{\label{F:example2}}
Reducibility of and equivalences between $V_\trip{c}$ for $V_\chi$ with $M=1$ and $N=2$.}
\end{caption}
\end{figure}

\begin{table}[t]
\begin{tabular}{|ll|}
\hline
Quotient & Dimension  \\ \hline
$V_{(1,4,4)}$ & $q^5(q-1)\alpha$ \\
$V_{(3,3,3)}, V_{(3,2,4)}$ & $q^4(q-1)^2\alpha$  \\
$V_{(2,3,4)}$ & $q^4(q-1)(q-2)\alpha$ \\
 $V_{(1,3,4)}, V_{(2,2,4)}, V_{(3,2,3)}$ & $q^4(q-1)\alpha$\\
$V_{(2,3,3)}$ & $q^3(q-1)^2\alpha$ \\
    $V_{(1,3,3)}$ & $q^3(q-1)\alpha$\\
$V_{(2,2,3)}$ & $q^2(q-1)^2\alpha$\\
$V_{(1,2,3)}, V_{(2,2,2)}$ & $q^2(q-1)\alpha$\\
$V_{(1,2,2)}$ & $q^2\alpha$ \\ \hline
\multicolumn{2}{c}{} \\
\end{tabular}
\begin{caption}{\protect{\label{T:example2}}
Dimensions of $V_\trip{c}$ for $V_\chi$ with $M=1$ and $N=2$.  }
\end{caption}
\end{table}

Consider a character $\chi$ for which
$M=1$ and $N=2$.  Figure~\ref{F:example2} describes a portion of the
restriction to $K$ of $V_\chi$, namely, all subrepresentations
$U_\trip{c}$ for which $V_\trip{c}$ has dimension of order $q^9$ or less.
In terms of triples, this implies that we consider the elements
$\trip{c} \in \T_\trip{m}$ for which $c_1+c_2+c_3 \leq 9$.
Again, the number of intertwining operators between each pair of quotients
is determined by Corollary~\ref{C:intops}.

This example illustrates a phenomenon not present
in Example~\ref{Ex:1}.
For instance, there are two pairs of isomorphic irreducible representations:
$V_{(1,3,4)} \simeq V_{(2,2,4)}$ (indicated by $*$ in Figure~\ref{F:example2})
and one of the two inequivalent irreducible summands of
$V_{(1,4,4)}$ is isomorphic to exactly one of the irreducible summands
of $V_{(2,3,4)}$ (indicated by $\dagger$ in Figure~\ref{F:example2}).
Such pairs of isomorphic irreducibles, for
distinct triples $\trip{c}, \trip{d} \in \T_\trip{m}$,
can occur only when $c_3=d_3$, since otherwise
the corresponding groups $C_\trip{c}$ and $C_\trip{d}$ lie
in different levels of the filtration of $K$ by the normal subgroups
$K_n$, which in turn would imply that $V_\trip{c}$
and $V_\trip{d}$ cannot intertwine as representations of $K$.

The dimensions of the representations in Figure~\ref{F:example2} are
given in Table~\ref{T:example2}. We have again abbreviated $\alpha = (q+1)(q^2+q+1)$.
We conjecture that $V_{(2,3,4)}$ in fact decomposes as a sum of
$q-2$ distinct irreducibles, each of dimension equal to that of
$V_{(1,3,4)}$, $V_{(2,2,4)}$ and $V_{(3,2,3)}$.
This would
be consistent with the remaining irreducible in $V_{(1,4,4)}$ having
dimension equal to that of $V_{(3,3,3)}$ and $V_{(3,2,4)}$.
\end{example}

\end{document}